\documentclass[12pt,leqno]{amsart}
\usepackage{fourier}
\usepackage{amsmath,mathtools, amssymb, amsfonts,mathrsfs,graphicx,amsmath,bbold}
\usepackage[mathscr]{eucal}
\usepackage[backgroundcolor=yellow,shadow,colorinlistoftodos]{todonotes}
\usepackage[hypertex]{hyperref}

\textheight9in  \def\DATE{\today}
\hoffset - 2cm  \hbadness=100000
\usepackage[all]{xy}

\makeatletter
\providecommand\@dotsep{5}
\def\listtodoname{List of Todos}
\def\listoftodos{\@starttoc{tdo}\listtodoname}
\makeatother

\catcode`\@=11 
\def\@evenfoot{\rule{0pt}{20pt}[\today] \hfill [{\tt \jobname.tex}]}
\def\@oddfoot{\rule{0pt}{20pt}{[\tt \jobname.tex}]\hfill [\today]}
\catcode`\@=13
\theoremstyle{plain}
\newtheorem{theorem}{Theorem}
\newtheorem*{theorem*}{Theorem}

\newtheorem{proposition}[theorem]{Proposition}
\newtheorem{lemma}[theorem]{Lemma}

\newtheorem*{lemma*}{Lemma}
\newtheorem{corollary}[theorem]{Corollary}

\newtheorem*{cor*}{Corollary}

\theoremstyle{definition}
\newtheorem{definition}[theorem]{Definition}

\newtheorem*{ass*}{Assumption}
\newtheorem*{conventions*}{Conventions}
\newtheorem*{acknowledgments*}{Acknowledgment}

\newtheorem*{notation*}{Notation}
\newtheorem{remark}[theorem]{Remark}

\theoremstyle{plain}

\newtheorem*{question*}{Question}

\def\Ass{\EuScript{A}{\it ss}}
\def\Se{{\!\!\tt St}}
\def\oPuS{\oP_{\underline S}}
\def\oPS{\oP_S}
\def\minimS{\minim_S}
\def\minimuS{\minim_{\underline S}}
\def\minimiso{\minim_{{\EuScript I}so}}
\def\End{\hbox{${\mathcal E}\hskip -.1em {\it nd}$}}
\def\sipka{{\B \to \W}}
\def\calF{{\mathcal F}}
\def\minim{{\mathfrak M}}
\def\malesipky{{%
{
\unitlength=.15pt
\B \hskip 2pt
\begin{picture}(80.00,40.00)(0.00,0.00)
\put(80.00,10.00){\vector(2,1){0}}
\put(80.00,30.00){\vector(2,-1){0}}
\put(36.00,20.00){\makebox(0.00,0.00){\line(0,1){20}}}
\put(44.00,20.00){\makebox(0.00,0.00){\line(0,1){20}}}
\bezier{50}(0.00,10.00)(40.00,-10.00)(80.00,10.00)
\bezier{50}(0.00,30.00)(40.00,50.00)(80.00,30.00)
\end{picture}}
\hskip 2pt \W
}}

\def\bbZ{{\mathbb Z}}
\def\bbk{{\mathbb k}}

\def\susp#1{\uparrow \hskip -.3em   #1}
\def\antishriek{{\scriptsize \raisebox{.2em} {!`}}}
\def\CCn#1#2{{C^{#1}_{\sscooP}}(#2)}
\def\CC#1{{C^c_{\sscooP}}(#1)}
\def\otexp#1#2{{#1^{\otimes #2}}}
\def\cooP{\raisebox{.8em}{\rotatebox{180}{$\EuScript P$}}}
\def\cogamma{\raisebox{.45em}{\rotatebox{180}{$\gamma$}}}
\def\sscooP{\raisebox{.7em}{\rotatebox{180}{\scriptsize $\EuScript P$}}}
\def\ssscooP{\raisebox{.5em}{\rotatebox{180}{\tiny $\EuScript P$}}}
\def\oP{{\EuScript P}}
\def\oA{{\EuScript A}}
\def\Rada#1#2#3{#1_{#2},\dots,#1_{#3}}
\def\Tc#1{\hbox{{$T^c(\uparrow \hskip -.3em   #1)$}}}
\def\E{{\tt E}}
\def\B{{\tt B}}
\def\W{{\tt W}}
\def\bfF{{F}}

\def\asim{{\ \stackrel{\scriptscriptstyle \infty}\thicksim \ }}
\def\ot{\otimes}
\def\Set{{\tt Set}}

\def\ttE{{\tt E}}
\def\iA{{\mathbb A}}\def\iB{{\mathbb B}}\def\iC{{\mathbb C}}
\def\ov#1{\redukcem{\overset{\lcirc}{#1}}}
\def\lcirc{{\raisebox{-.2em}{\scriptsize $\circ$}}}
\def\map#1{(#1,#1_2,#1_3,\ldots)}
\def\Umap#1{#1,#1_2,#1_3,\ldots}
\def\Ainfty{\hbox{${A_\infty}$}}
\def\CHE{{\tt CHE}}
\def\Iso{{\tt Isot}_\infty}
\def\id{{\rm id}}
\def\id{\mathbb 1}
\def\pa{\partial}

\def\bfmu{{\boldsymbol \mu}}

\def\bfomega{{\boldsymbol \omega}}
\def\bfnu{{\boldsymbol \nu}}
\def\redukce#1{\vbox to .3em{\vss\hbox{#1}}}
\def\redukcem#1{\vbox to .3em{\vss\hbox{$#1$}}}
\def\cP{\oP}

\begin{document}

\title[Transfers as bifibrations]{Transfers of $A_\infty$- and other
  homotopy structures as Grothendieck bifibrations}

\author{Martin Markl}
\address{The Czech Academy of Sciences, Institute of Mathematics, {\v Z}itn{\'a} 25,
         115 67 Prague, The Czech Republic}
\email{markl@math.cas.cz}
\thanks{Supported by Praemium Academi\ae\ and RVO: 67985840.}

\subjclass[2000]{16E99, 55S20}
\keywords{$A_\infty$-algebra, transfer, isotopy, bifibration}

\begin{abstract}
We show that the functor which assigns to an \Ainfty-morphism between isotopy
classes of \Ainfty-algebras whose linear part is a chain homotopy
equivalence its underlying chain map
is a~discrete Grothendieck bifibration.   We
then generalize our results to $\cP_\infty$-structures over a field
of characteristic zero, for any quadratic Koszul operad $\cP$. An
immediate application is a categorical framework in which the transfers of
e.g.~$A_\infty$-, $L_\infty$- and $C_\infty$-structures are strictly
  functorial.
A by-product of our reasoning is a general transfer theorem for
$\oP_\infty$ algebras, which we prove in the last section.
\end{abstract}

\baselineskip 15pt plus 1pt minus 1pt
\maketitle


\setcounter{tocdepth}{1}
\tableofcontents

\section*{Introduction}

The central theme of the joint paper~\cite{sull} of C.~Rogers and the author
was the question of how, given 
an \Ainfty-algebra $ (A,\pa,\mu_2,\mu_3,\ldots)$ and an
\Ainfty-morphism 
\[
F = (f,f_2,\ldots) :  (A,\pa,\mu_2,\mu_3,\ldots) \longrightarrow
(B,\pa,\nu_2,\nu_3,\ldots),
\] 
the $A_\infty$-structure $\bfnu = \nu_2,\nu_3,\ldots$ on the chain
complex $(B,\pa)$ depended
on the linear part $f$ of $F$. In loc.~cit.\ we 
called the \Ainfty-structure $\bfnu$ above a {\em transfer of\/} $\bfmu =
\mu_2,\mu_3,\ldots$  {\em  over the chain map} $f$.

There are many situations where a transfer exists. This
happens, e.g.,~when $f$ admits a left homotopy inverse~\cite{tr}. 
If $f$ induces an isomorphism of homology, and 
suitable freeness or projectivity 
assumptions are satisfied, the transferred structure
can be constructed inductively via homological obstruction theory, as was
done in the seminal paper by T.~Kadeishvili~\cite{kadeishvili:RMS80}.
A thorough historical account of the later approach can be found in~\cite{pet}.

By~\cite[Theorem~2]{sull}, the transfers over a chain map which
is a chain homotopy
equivalence are {\em unique up to an
isotopy\/} which is, by definition, an \Ainfty-automorphism whose
linear part is the identity map. This implies the following
`functoriality up to isotopy:'  
Consider the composite  $h := gf : (A,\pa) \to (C,\pa)$ of  
chain homotopy equivalences
$f: (A,\pa) \to (B,\pa)$ and $g: (B,\pa) \to (C,\pa)$. 
Let $\bfnu$ be a transfer of a given \Ainfty-structure $\bfmu$
on $(A,\pa)$ over~$f$, $\bfomega'$ a transfer of $\bfnu$ over
$g$, and $\bfomega''$  a transfer of $\bfmu$ over the
composite $h$. Then the \Ainfty-algebras $(C,\pa,\bfomega')$ and 
 $(C,\pa,\bfomega'')$ are isotopic. In fact, even a somewhat stronger
 result formulated in~Corollary~\ref{Dnes jsme byli s Jarkou na CEZu.}
below holds.

This raises the natural question of whether there is an appropriate
category of isotopy
classes of $A_\infty$-algebras such that the transfers taking
values in this hypothetical category  
will be {\em strictly\/} functorial. 
In the present note we show that a suitable category of isotopy
classes, denoted $\Iso$, indeed exists. 
Its existence will follow from Proposition~\ref{Pojedu
  tento tyden do Mercina?}, which is the main technical and
surprisingly profound result of this
note. 
Proposition~\ref{Pojedu tento tyden do Mercina?} uses
Lemma~\ref{Vcera jsem ten kopec na Tocnou vyjel.}. Its conclusion 
could be anticipated, but we found a surprisingly easy and
somewhat explicit proof. That is why we have put it in a separate
section with a pompous name.

The category $\Iso$ admits a forgetful functor $\Box : \Iso \to \CHE$
to the category of chain complexes and their chain homotopy
equivalences. The above-mentioned functoriality of transfers is expressed by
the fact that  $\Box : \Iso \to \CHE$ is a discrete Grothendieck
bifibration, i.e.\ a functor which is both a discrete fibration and
opfibration. It is well-known that there is a one-to-one
correspondence between discrete (op)fibrations and functors
from the base category to the category of sets. The `total space' of
a given (op)fibration then appears as the category of elements of the
corresponding functor. This description will be the subject of the
last part of this note.

\vskip .5em
\noindent 
{\bf Acknowledgment:} 
I would like to thank Chris Rogers for his kind and useful comments on
the first version of this note, namely on the content of
Lemma~\ref{Vcera jsem ten kopec na Tocnou vyjel.}.  I am also indebted 
to Jos\'e Manuel Moreno for  suggestions which  
led to a substantial improvement in the final version.

\vskip .5em
\noindent 
{\bf Conventions.} 
All algebraic objects in this note will be defined over an arbitrary
unital commutative ring $R$. In particular, they may live in the
category of abelian groups. An exception is
Section~\ref{Medvidek s Micinkou} where we need to assume at some
places that $R$ is a field of characteristic $0$. 
By $\ot$ we denote the tensor product
over $R$ and by $\id_X$  the identity
automorphism of an $R$-module $X$. 
Quasi-isomorphism will be abbreviated to quism.

For \Ainfty-algebras, their
morphisms and homotopies between morphisms we use the degree and 
sign conventions specified
in~\cite[page 141]{tr}. Namely, all differentials will be of degree
$-1$. Chain complexes, possibly unbounded, will be typically denoted by 
$(A,\pa)$, $(B,\pa)$, \&c. We will use the
same symbol $\pa$ for all  differentials, since they will always be
determined by their underlying spaces. An \hbox{$A_\infty$-algebra} $A =
(A,\pa,\mu_2,\mu_3,\ldots)$  will be abbreviated by $A = (A,\pa,\bfmu)$,
with  $\bfmu$ the collective symbol for the  higher products $\mu_k:
A^{\ot k } \to A$, $k \geq 2$.
An \Ainfty-morphism 
\begin{equation}
\label{Pozitri na kontrolu.}
\bfF = (f,f_2,f_3,\ldots) : (A,\pa,\mu_2,\mu_3,\ldots) \longrightarrow
(B,\pa,\nu_2,\nu_3,\ldots),\ f : A \to B,\
f_k:A^{\ot k} \to B,\ k \geq 2,
\end{equation}
will be abbreviated accordingly by $F :  (A,\pa,\bfmu) \to
(B,\pa,\bfnu)$. Such a morphism $F$ will be 
called an {\em isotopy\/} if its linear part
$F_1 := f$ equals  the identity automorphism $\id_A$ of $A$. 
Then of course $(B,\pa) = (A,\pa)$. 
The \Ainfty-morphism in~(\ref{Pozitri na kontrolu.}) 
is an {\/\em extension\/} of the chain map $f: (A,\pa) \to (B,\pa)$.

\vskip .5em

\noindent 
{\bf A reminder on transfers.}
The following main result of~\cite{tr} and its simple corollary  will
be frequently used in this note .

\begin{theorem}
\label{Stal se ten zazrak?}
Let $f : (A,\pa) \to (B,\pa)$ be a chain map, $g : (B,\pa) \to
(A,\pa)$ its left chain homotopy inverse, and $h$ a chain homotopy
between $gf$ and the identity automorphism of $A$, i.e.\ there exists a diagram
\[
\xymatrix@C=5em{
 *{\quad \quad  (A,\pa) \ } \ar@(ul,dl)[]_{h}}
\hskip -3.8em
\xymatrix@C=5em{
\rule{3em}{0em}\rule{0em}{2em}  \ar@<0.2em>@/^.8em/[r]^f
& {(B,\pa)\, ,} \ar@<0.2ex>@/^.8em/[l]_g 
}\
gf - \id_A = \pa h + h \pa. 
\]
Then any \Ainfty-algebra structure $(A,\pa,\bfnu)$ on $(A,\pa)$ induces
an \Ainfty-algebra structure $(B,\pa,\bfnu)$ on
$(B,\pa)$ such that the chain maps $f$ and $g$ extend to
\Ainfty-morphisms $F$ and~$G$, and the chain homotopy $h$ extends to
an \Ainfty-homotopy $H$ between $GF$ and $\id_A$, so we have the
diagram
\[
\xymatrix@C=5em{
 *{\quad \quad \quad (A,\pa,\bfmu) \hskip .3em} \ar@(ul,dl)[]_{H}}
\hskip -3.8em
\xymatrix@C=5em{
\rule{3.3em}{0em}\rule{0em}{2.2em}  \ar@<0.2em>@/^.7em/[r]^F
& {(B,\pa,\bfnu)\, .} \ar@<0.2ex>@/^.7em/[l]_G 
}
\]
\end{theorem}

\begin{corollary}
\label{hruba verze}
Let $f : (A,\pa) \to (B,\pa)$ be a chain homotopy equivalence and
$(A,\pa,\bfmu)$ an \Ainfty-algebra structure on $(A,\pa)$. Then there
exists an \Ainfty-structure $(B,\pa,\bfnu)$ on $(B,\pa)$ and an
extension of $f$  to an
\Ainfty-morphism
$F :   (A,\pa,\bfmu) \longrightarrow  (B,\pa,\bfnu)$.
\end{corollary}

\section*{Layout of the paper}

Sections~\ref{Krtek a Laurinka}--\ref{dva kasparci a certik Pik} are devoted
to $A_\infty$-algebras and do not assume any knowledge of operads.
Section~\ref{Medvidek s Micinkou} in which the results are generalized
to $\oP_\infty$-algebras requires knowledge of operad theory; we refer to the monograph
\cite{MSS} or a more 
recent~\cite{loday-vallette}. It also relies on the apparatus
of~\cite{ws01}. The range of applicability of the results
is summarized at the very end of the main text.

\section{Videte miraculum}
\label{Krtek a Laurinka}

All algebraic objects here will be defined over an arbitrary
commutative unital ring, in particular, over the ring $\bbZ$ of
integers.  The following lemma will be used in the
proof of Proposition~\ref{Pojedu tento tyden do Mercina?}.

\begin{lemma}
\label{Vcera jsem ten kopec na Tocnou vyjel.}
Let $F :(A',\pa,\bfmu') \to (A'',\pa,\bfmu'')$ be an
\Ainfty-morphism and $g:(A',\pa) \to (A'',\pa)$ a~chain map, chain
homotopic to the linear part $f$ of $F = \map f$. Then $g$ can be extended to an
\Ainfty-morphism
\[
G = \map g : (A',\pa,\mu_2',\mu_3',\ldots) \longrightarrow
(A'',\pa,\mu_2'',\mu_3'',\ldots) 
\] 
which is, moreover,  \Ainfty-homotopic to $F$. 
\end{lemma}

A special case of the lemma above, valid over a field of
characteristic zero, was given in~\cite[Proposition~35]{haha} using
different methods. The existence of the \Ainfty-morphism $G$ was the
content of~\cite[Proposition~14]{sull}; the \Ainfty\ homotopy between
$G$ and $F$ was
implicitly constructed in its proof. Below we offer a surprisingly
simple, direct

\begin{proof}
It is well-known that \Ainfty-algebras are the same as
differentials on tensor coalgebras (i.e.~degree $-1$ coderivations
that square to zero), 
cf.~\cite[Example~II.3.90]{MSS}. This correspondence essentially uses the fact
that tensor coalgebras are cofree conilpotent coassociative
coalgebras. 
In this language, \Ainfty-morphisms are
coalgebra morphisms commuting with the differentials, and 
\Ainfty-homotopies  appear as
coderivation homotopies between \hbox{coalgebra morphisms}.

More specifically, the family $\pa,\mu_2',\mu_3',\ldots$ of
multilinear maps
determines a degree $-1$ coderivation
$\delta'$ on the tensor coalgebra
 $\Tc {A'}$ cogenerated by the suspension of $A'$. 
The \Ainfty-axioms for the higher products are 
equivalent to \hbox{$\delta'^2=0$}. 
Similarly, $\pa,\mu_2'',\mu_3'',\ldots$ 
determines  a coderivation~$\delta''$ on
$\Tc {A''}$ satisfying \hbox{$\delta''^2=0$}.  

The family  
$F =\map f$ induces  a coalgebra morphism  $\phi : \Tc {A'} \to \Tc
{A''}$,  and $F$ is an
\Ainfty-morphism if and only if 
$\delta'' \phi = \phi \delta'$. Analogically, a family
$G=\map g$ is an \Ainfty-morphism if and only if the induced
coalgebra morphism  $\psi : \Tc {A'} \to \Tc{A''}$ satisfies 
\hbox{$\delta'' \psi = \psi \delta'$}. 
An \Ainfty-homotopy between $F$ and $G$ is 
given by a family $H =\map h$ of linear~maps 
\[
h : A' \to A'',\ h_k : A'^{\ot k} \to A'',\ \deg(h) = 1,\ 
\deg (h_k) = k, \ \hbox { for } \  k \geq 2.
\]
Every such a family assembles to a coderivation homotopy $\eta :  \Tc {A'} \to
\Tc {A''}$ which is, by definition, a degree $+1$ linear map satisfying
\[
\Delta'' \circ \eta = (\phi \ot \eta) \circ \Delta' +  
(\eta \ot \psi) \circ \Delta',
\]
where $\Delta'$ resp.~$\Delta''$ are the coproducts of $\Tc {A'}$
resp.~$\Tc {A''}$. 
The family $H = \map h$ is then an \Ainfty-homotopy between $G$
and $F$ if and only if 
\begin{equation}
\label{Za chvili jdu s Janou do MATu.}
\psi - \phi=  \delta'' \eta + \eta  \delta'
\end{equation}
The above translation of the \Ainfty-notions to 
the language of differential graded coalgebras is summarized in the following table:
\begin{subequations}
\begin{align*}
\hbox{$(A',\pa,\mu'_2,\mu'_3,\ldots)$ is an \Ainfty-algebra} 
& \ \Longleftrightarrow \   \delta'^2=0,
\\
\hbox{$(A'',\pa,\mu''_2,\mu''_3,\ldots)$ is an \Ainfty-algebra} 
& \ \Longleftrightarrow \   \delta''^2=0,
\\
\hbox{$F = \map f$ is an \Ainfty-morphism} 
& \ \Longleftrightarrow \   \delta'' \phi=\phi \delta',
\\
\hbox{$G = \map g$ is an \Ainfty-morphism} 
& \ \Longleftrightarrow \   \delta'' \psi=\psi \delta', \ \hbox { and}
\\
\hbox {$H = \map h$ is an \Ainfty-homotopy between $G$ and $F$}
& \ \Longleftrightarrow \ \psi - \phi=  \delta'' \eta + \eta  \delta'.
\end{align*}
\end{subequations}
Let us return to equation~\eqref{Za chvili jdu s Janou do MATu.}. 
In terms of the families of (multi)linear maps
\[
\mu'_2,\mu'_3,\ldots,\   
\mu''_2,\mu''_3,\ldots,\ 
\Umap f,\ \Umap g  \hbox { and } \Umap h,
\]
it is equivalent to the following infinite system of equations:
\begin{gather*}
\tag{$H_n$}
g_n  =  \ \pa h_n -(-1)^n h_n \pa +f_n
+ \sum_A \pm
\mu''_k\big(g_{r_1} \ot
\cdots \ot
g_{r_{i-1}}  \ot  h_{r_i}  \ot f_{r_{i+1}} \ot \cdots \ot f_{r_k}\big)
\\
\nonumber 
+ \sum_B \pm
h_k (\id_{A'}^{ \ot i-1} \ot \mu'_l \ot \id_{A'}^{\ot k-i}),\
n \geq 1.
\end{gather*}
In the above display, 
\begin{eqnarray*}
A\hskip -2mm &:=& \hskip -2mm \{k,i,\Rada r1k\ |\   
       2 \leq k \leq n,\ 1 \leq i \leq k,\ r_1,\ldots, r_k \geq 1,\
r_1+ \cdots + r_k = n\},
\\
B \hskip -2mm &:=& \hskip -2mm \{k,l\ |\ k+l = n+1,\ k,l \geq 2,\ 
1 \leq i \leq k\},
\end{eqnarray*}
and $f_1:= f$, $g_1 := g$. We intentionally left some  signs out, 
since only the types of the terms will matter. 
A zealous reader might find the explicit signs in Section~2 of~\cite{tr}.

Let $h$ be a given homotopy between $g$ and $f$, and  $h_k :
A'^{\otimes k} \to A''$, $k \geq 2$, {\/\em arbitrary\/} linear maps
of degree $k$.
We will construct the \Ainfty-morphism $G = \map g$ inductively, 
component by component, such that the homotopy between
$G$ and $F$ will be $H = \map h$.
To help to understand our construction better, we write explicitly the initial
three instances of ($H_n$). The result, explicit signs included, reads
\begin{align*}
\tag{$H_1$}
g  =& \  \pa h + h\pa +  f, 
\\  
\tag{$H_2$}
g_2 = &\   \pa h_2 - h_2\pa  + f_2  - \mu''_2 (g \ot h) - \mu''_2(h\ot f) +
        h\mu'_2\, ,\  \hbox { and}
\\
\tag{$H_3$}
g_3 = &\ \pa h_3  + h_3\pa + f_3 +
\mu''_3(g^{\ot 2} \ot h) + \mu''_3(g \ot h \ot f) + \mu''_3(h \ot
                   f^{\ot 2} ) -\mu''_2(g_2 \ot h) - \mu''_2(h \ot f_2)
\\
   &  
+\mu''_2(g\ot h_2) + \mu''_2(h_2 \ot f)- h\mu'_3 - h_2(\mu'_2 \ot
     \id_{A'}) + h_2(\id_{A'}  \ot \mu'_2)\, .
\end{align*}

Equation ($H_1$) holds by the assumption of the lemma. All
terms at the right hand side of ($H_2$) are defined, so ($H_2$) determines $g_2$.
Having $g_2$, all terms in the right hand side of ($H_3$)
have been defined, so ($H_3$) determines $g_3$. The induction goes
ahead in the obvious way, producing a
family $G =\map g$ satisfying ($H_n$) for each $n$\,.
The extension of that family to a coalgebra
morphism $\psi:\Tc{A'} \to \Tc{A''}$ satisfies~(\ref{Za
  chvili jdu s Janou do MATu.}), because this is precisely what
equations ($H_n$)~mean.

It remains to show that the family
$G =\map g$ represents an \Ainfty-morphism. This is, by the above
observations, equivalent to $\delta'' \psi = \psi \delta'$. Here comes
the advertised miracle.
Rewriting~(\ref{Za chvili jdu s Janou do MATu.}) as
$\psi = \phi +\delta'' \eta + \eta \delta'$,
 $\delta'' \psi = \psi \delta'$ equivalent to 
\[
\delta''(\phi +\delta'' \eta + \eta \delta') = (\phi +\delta'' \eta + \eta \delta')\delta'
\]
which is the same as
\[
\delta''\phi +\delta''\delta'' \eta + \delta''\eta \delta' 
= \phi\delta' +\delta'' \eta\delta' + \eta \delta'\delta'
\]
which holds, since $\delta''\phi = \phi\delta'$.
\end{proof}

\begin{corollary}
\label{Bude dnes opravene kolo?}
Let $E = \map e : (X,\pa,\bfmu') \to (X,\pa,\bfmu'' )$ be an
\Ainfty-endomorphism whose linear part $e$ is chain homotopic to the identity
$\id_X : (X,\pa) \to (X,\pa)$. Then $E$  is \Ainfty-homotopic
to  an isotopy ${T}: (X,\pa,\bfmu') \to
(X,\pa,\bfmu'')$.
\end{corollary}

\section{The category of isotopy classes}
\label{Kralicek a Oslicek}

In the following, we express with \hbox{$f'\! \sim \! f''$} that the chain map $f'$
is chain homotopic to $f''$. The notation 
$F' \stackrel{\scriptscriptstyle \infty}\thicksim F''$ means that
the $A_\infty$-morphisms $F'$ and $F''$ are  $A_\infty$-homotopic.
Here is the main technical result of the note:

\begin{proposition}
\label{Pojedu tento tyden do Mercina?} 
Let $F' : (A,\pa,\bfmu') \to  (B,\pa,\bfnu')$ and
$F'' : (A,\pa,\bfmu'') \to  (B,\pa,\bfnu'')$
be $\Ainfty$-morphisms with the underlying chain maps $f',
f'' : (A,\pa) \to  (B,\pa)$ which are chain homotopy equivalences.
Then the following are equivalent: 
\begin{itemize}
\item [(i)] 
$f' \! \sim \! f''$, and $(A,\pa,\bfmu')$ and $(A,\pa,\bfmu'')$ are
isotopic,
\item [(ii)] 
$f' \! \sim \! f''$, and $(B,\pa,\bfnu')$ and $(B,\pa,\bfnu'')$ are
isotopic,
\item [(iii)]
$f' \! \sim \! f''$, $(A,\pa,\bfmu')$ and $(A,\pa,\bfmu'')$ are
isotopic, and also  $(B,\pa,\bfnu')$ and $(B,\pa,\bfnu'')$ are isotopic,
\item [(iv)]
There are isotopies $S:(A,\pa,\bfmu') \to (A,\pa,\bfmu'')$ and 
$T:(B,\pa,\bfnu') \to (B,\pa,\bfnu'')$ such that the diagram
\begin{equation*}
\xymatrix@C=.5em@R=.5em{(A,\pa,\bfmu') \ar[rr]^{F'} \ar[dd]_{S}   
&&  (B,\pa,\bfnu') \ar[dd]^{T}
\\
&\asim&
\\
(A,\pa,\bfmu'')  \ar[rr]^{F''}   && (B,\pa,\bfnu'')
}
\end{equation*}
commutes up to $\Ainfty$-homotopy.
\end{itemize}
\end{proposition}

\begin{proof} 
Since $\hbox{(iv)} \Longrightarrow \hbox{(iii)} 
\Longrightarrow \hbox{(i)}\ \&\ \hbox{(ii)}$, it suffices to prove that
both (i) or (ii) imply~(iv).
Let us start with \hbox{(i) $\Longrightarrow$ (iv)}.
For this purpose, consider the diagram
\begin{equation}
\label{Musim se jit vecer probehnout.}
\xymatrix@C=4em{(A,\pa,\bfmu') \ar[rr]^{F'} \ar[dd]_{S}
\ar@{-->}@/^/[rrd]^(.65){\overset{\lcirc}{F}}
 &  &  (B,\pa,\bfnu') 
\\
&&  (B,\pa,\overset{\lcirc}{\bfnu}) \ar@{-->}@/^2em/[llu]^(.5){\overset{\lcirc}{G}}
\ar@{-->}[u]_{T'}\ar@{-->}[d]^{T''}
\\
(A,\pa,\bfmu'')  \ar[rr]^{F''}   && (B,\pa,\bfnu'')
}
\end{equation}
in which $S: (A,\pa,\bfmu') \to (A,\pa,\bfmu'')$ 
is an  isotopy. Theorem~5 of~\cite{tr}
guarantees the existence of \Ainfty-morphisms $\redukcem{\overset{\lcirc}{F}}$
and $\redukcem{\overset{\lcirc}{G}}$ such that  $\redukcem{\overset{\lcirc}{F}}$ extends
$f'$, the linear part  $\ov{g}$ of $\redukcem{\overset{\lcirc}{G}}$ 
is a chain homotopy inverse of $f'$, 
and  $\redukcem{\overset{\lcirc}{G}\overset{\lcirc}{F}} \asim  \id_A$.

Now $F' \redukcem{\overset{\lcirc}{G}} :  (B,\pa,\overset{\lcirc}{\bfnu}) \to
(B,\pa,\bfnu')$ is an \Ainfty-map with the linear part $f'\ov{g}$
chain homotopic to $\id_B$ thus, 
by Corollary~\ref{Bude dnes opravene kolo?} above, 
there exists an isotopy $T': (B,\pa,\redukcem{\overset{\lcirc}{\bfnu}}) \to
(B,\pa,\bfnu')$, \Ainfty-homotopic to
$F' \redukcem{\overset{\lcirc}{G}}$. Since $S$ is an isotopy, it does not affect
the linear parts of \Ainfty-maps, so the linear part of $F'' S \ov
G$ is~$f'' \ov g$, which is chain homotopic to
$f' \ov g \sim \id_B$ since $f' \sim f''$ by assumption. 
Corollary~\ref{Bude dnes opravene kolo?} thus gives  an
isotopy  $T'': (B,\pa,\overset{\lcirc}{\bfnu}) \to
(B,\pa,\bfnu'')$, \Ainfty-homotopic to
$F''  S\redukcem{\overset{\lcirc}{G}}$.  The above 
morphisms fulfill  the hypotheses of the obvious implications
\begin{subequations}
\begin{align}
\label{1}
F''S\overset{\lcirc}{G} \asim T''\ \& \   \   \overset{\lcirc}{G}
\overset{\lcirc}{F} \asim \id_A  \ &\Longrightarrow \ F''S \asim T'' \ov{F},
\  \hbox { and }
\\
\label{2}
F'\overset{\lcirc}{G} \asim T'\ \& \   \   \overset{\lcirc}{G}
\overset{\lcirc}{F} \asim \id_A  \ &\Longrightarrow \ F' \asim T'
\overset{\lcirc}{F}.
\end{align}
\end{subequations}
We claim that (iv) is satisfied with $T := T'' T'^{-1}$. Indeed, the
homotopy commutativity 
\[
F'' S \asim T''  T'^{-1} F' = TF'
\] 
of~(\ref{Musim se jit vecer probehnout.}) is equivalent to $T''
\ov F \asim T''  T'^{-1} F'$ by~(\ref{1}). Applying $T''^{-1}$ to both sides from the
left, we see that the latter is equivalent to $\ov F \asim  T'^{-1}
F'$, which is in turn equivalent to $T' \ov F \asim F'$, established
in~(\ref{2}).

The proof of \hbox{(ii) $\Longrightarrow$ (iv)} is similar, so we only
indicate the main steps.  Let $g$ be a homotopy inverse
of $f'$ which is, of course, simultaneously a homotopy inverse of
$f''$.  Theorem~5 of~\cite{tr} gives $A_\infty$-maps  $\overset{\lcirc}{G} :(B,\pa,\bfnu'')\to
(A,\pa,\overset{\lcirc}{\bfmu})$ and 
$\overset{\lcirc}{F} : (A,\pa,\overset{\lcirc}{\bfmu}) \to (B,\pa,\bfnu'')$
in the diagram
\[
\xymatrix@C=4em{(A,\pa,\bfmu') \ar[rr]^{F'} \ar@{-->}[d]_{S'}
&&  (B,\pa,\bfnu') \ar[dd]^T
\\
(A,\pa,\overset{\lcirc}{\bfmu}) \ar@{-->}@/_/[rrd]^(.5){\overset{\lcirc}{F}}&& 
\\
(A,\pa,\bfmu'')  \ar[rr]^{F''} \ar@{-->}[u]_{S''} 
&& (B,\pa,\bfnu'') \ar@{-->}@/_1.8em/[llu]_(.5){\overset{\lcirc}{G}}
}
\]
so that ${\overset{\lcirc}{F}}{\overset{\lcirc}{G}} \asim \id_B$.
The linear part  $g
f'$ of the $A_\infty$-map
${\overset{\lcirc}{G}} T {\overset{\lcirc}{F}}$ is homotopic to $\id_A$, 
so Corollary~\ref{Bude dnes opravene kolo?} gives  an
isotopy  $S': (A,\pa,\bfmu') \to
(A,\pa,\overset{\lcirc}{\bfmu})$, \Ainfty-homotopic to
$\redukcem{\overset{\lcirc}{G}}T F'$. Similarly we get an isotopy
$S'': (A,\pa,\bfmu'') \to (A,\pa,\overset{\lcirc}{\bfmu})$. Then
$S:= S''^{-1}S'$ is the isotopy required in~(iv).
\end{proof}

Proposition~\ref{Pojedu tento tyden do Mercina?} implies the
functoriality of transfers up to isotopy, formulated in

\begin{corollary}
\label{Dnes jsme byli s Jarkou na CEZu.}
Consider a homotopy commutative diagram of chain homotopy equivalences
\[
\xymatrix{
& (B,\pa) \ar@/^/[rd]^g&
\\
(A,\pa)   \ar@/^/[ru]^f   \ar[rr]^{h}  && (C,\pa)
}
\]
along with the data consisting of
\begin{itemize}
\item 
an \Ainfty-structure $(A,\pa,\bfmu)$ on $(A,\pa)$,
\item 
an extension $F = (f,f_2,f_3,\ldots) :
(A,\pa,\bfmu) \to (B,\pa,\bfnu)$ of $f$,
\item
an extension $G = (g,g_2,g_3,\ldots) :
(B,\pa,\bfnu) \to (C,\pa,\bfomega')$ of $g$ and
\item
an extension  $H = (h,h_2,h_3,\ldots) :  
(A,\pa,\bfmu) \to (C,\pa,\bfomega'')$ of the composite $h :=
gf$.
\end{itemize} 
Then there exists an isotopy $T: (C,\pa,\bfomega') \to
(C,\pa,\bfomega'')$ such that the diagram of \Ainfty-morphisms
\[
\xymatrix@R=.5em{
& (B,\pa,\bfnu) \ar[r]^G&(C,\pa,\bfomega') \ar[dd]^T
\\
&\asim&
\\
(A,\pa,\bfmu)   \ar@/^2em/[ruu]^F   \ar[rr]^{H}  && (C,\pa,\bfomega'')
}
\]
commutes up to \Ainfty-homotopy. 
\end{corollary}

\begin{proof}
The assumptions of the corollary are summarized in the diagram 
\[
\xymatrix@R=1.5em{
(A,\pa,\bfmu) \ar@{=}[dd]  \ar[r]^F 
& (B,\pa,\bfnu) \ar[r]^G&(C,\pa,\bfomega')
\\
&&
\\
(A,\pa,\bfmu)    \ar[rr]^{H}  && (C,\pa,\bfomega'')
}
\]
in which the linear part of the composite $GF$ is homotopic to the
linear part of $H$. We are therefore in the situation of item (i) of
Proposition~\ref{Pojedu tento tyden do Mercina?}, with the isotopy
being the identity endomorphism of $(A,\pa,\bfmu)$. It is easy
to check that the proof of \hbox{(i) $\Longrightarrow$ (ii)}
in this particular case leads to a diagram in item (iv) with $S =\id_A$.
\end{proof}

Let $\Iso$ be the category whose objects are isotopy classes of
$\Ainfty$-algebras. The arrows (i.e.~morphisms) in  $\Iso$ are equivalence classes of
$\Ainfty$-morphisms $F : (A,\pa,\bfmu) \to  (B,\pa,\bfnu)$ whose
underlying chain maps $f$ are chain homotopy equivalences,
modulo the relation which identifies $F'$ with $F''$ if and only if one
(and hence all) of conditions (i)--(iv)  in
Proposition~\ref{Pojedu tento tyden do Mercina?}  is satisfied. The
composition is defined as follows.

Assume that $\iA$, $\iB$ and $\iC$ are isotopy classes of
\Ainfty-algebras and $\phi : \iA \to \iB$, resp.~$\psi : \iB \to \iC$ two
arrows in $\Iso$. Assume that $\phi$ is the equivalence class of
some $F : (A,\pa,\bfmu) \to  (B,\pa,\bfnu')$ and $\psi$ the class
of $Y : (B,\pa,\bfnu'') \to  (C,\pa,\bfomega)$. Since 
the objects of $\Iso$ are isotopy classes, there exists
an isotopy $S: (B,\pa,\bfnu') \to  (B,\pa,\bfnu'')$. We define the
composite $\psi\phi : \iA \to \iC$ as the equivalence class of the
composed \Ainfty-morphism
\[
\xymatrix@R=2em{ (A,\pa,\bfmu) \ar[r]^F &  (B,\pa,\bfnu') \ar[d]^S
\\
&  (B,\pa,\bfnu'') \ar[r]^Y  &(C,\pa,\bfomega).
}
\]   
Let us prove that the equivalence class of $Y S F: (A,\pa,\bfmu) \to (C,\pa,\bfomega)$ does not depend on~$S$. Consider the diagram
\[
\xymatrix@R=1.7em{ 
&  (B,\pa,\bfnu'') \ar[rr]^Y  &&(C,\pa,\bfomega) \ar@{-->}[dd]^T
\\
(A,\pa,\bfmu) \ar[r]^F &  (B,\pa,\bfnu') \ar[d]^{S'}  \ar[u]_{S''} &\asim&
\\
&  (B,\pa,\bfnu'') \ar[rr]^Y  &&(C,\pa,\bfomega)
}
\]
in which $S'$ and $S''$ are two such isotopies. Then $S'{S''}^{-1}
:(B,\pa,\bfnu') \to  (B,\pa,\bfnu'')$ is
an isotopy too, thus Proposition~\ref{Pojedu tento tyden do Mercina?}
gives an isotopy $T : (C,\pa,\bfomega) \to (C,\pa,\bfomega)$ making
the square in the above diagram \Ainfty-homotopy commutative. Thus also the
square
\[
\xymatrix@C=.5em@R=.5em{(A,\pa,\bfmu) \ar[rr]^{YS''F} \ar@{=}[dd]
&&  (C,\pa,\omega) \ar[dd]^{T}
\\
&\asim
\\
(A,\pa,\bfmu)  \ar[rr]^{YS'F}   && (B,\pa,\bfomega)
}
\] 
is \Ainfty-homotopy commutative, so $YS'F$ and $YS''F$ belong to the
same equivalence class. We leave the analogous verification that the composite
$\psi\phi$ does not depend on the choices of the representatives of
$\phi$ and $\psi$ as an exercise.

\section{Transfers as bifibrations}
\label{dva kasparci a certik Pik}

In this section we study the functor that assigns to an isotopy
class of an \Ainfty-algebra its underlying chain complex. Let us
recall the following classical

\begin{definition}
\label{svitici snehulacek}
A functor $p: \E \to \B$ is a {\/\em discrete\/}
(Grothendieck) {\/\em fibration\/} if 
for each object $e''$ of the category $\E$ and each
morphism  $\beta : b'
\to p(e'')$ in $\B$ there exists a~unique
morphism $\varepsilon : e' \to e''$ in~$\E$ such that 
$p(\varepsilon) = \beta$. 

A functor is a {\/\em discrete opfibration\/} if the induced functor
between the opposite categories is a discrete fibration. Finally, a
functor is a {\/\em discrete bifibration\/} if it is both a discrete
fibration and opfibration. The lifting property 
is captured in:
\begin{equation*}
\xymatrix@C=1em@R=1em{e' \ar@{|->}[dd]_p  \ar@{-->}[rr]^\varepsilon &&  
e'' \ar@{|->}[dd]^p
\\
&\hbox {\tt fibration}
\\
b'   \ar[rr]^\beta    & & p(e'')
}
\hskip 2em
\xymatrix@C=.8em@R=1em{e' \ar@{|->}[dd]_p  \ar@{-->}[rr]^\varepsilon &&  
e'' \ar@{|->}[dd]^p
\\
&\hbox {\tt opfibration}
\\ 
p(e')   \ar[rr]^\beta    & & b''
}
\end{equation*}
\end{definition}

The adjective ``discrete'' in the definition means that the lifts are
unique. The general definition requires instead only a certain
universal property.

Denote by $\CHE$ the category whose objects are chain complexes and
morphisms are homotopy classes of chain maps which have
(unspecified) chain homotopy inverses. 
Notice that isotopic \Ainfty-algebras have
the same underling chain complex. Moreover, the underlying chain maps  
of all representatives of an
arrow in $\Iso$ are chain homotopic to each other. One therefore
has an obvious forgetful functor $\Box : \Iso \to \CHE$.
The main result of this section is

\begin{theorem}
\label{Uz druhy vikend prsi.}
The functor \/ $\Box : \Iso \to \CHE$ is a
surjective-on-objects discrete
Grothendieck bifibration.   
\end{theorem}

\begin{proof}
Since every chain complex carries the trivial \Ainfty-structure with
all higher products identically zero, the functor  $\Box$ is surjective
on objects as claimed.

We will look first at the more complicated opfibration case. The case of
fibrations is
`covariant' and therefore
simpler. Definition~\ref{svitici snehulacek} requires that, for the homotopy
class $[f]$ of a chain homotopy equivalence $f: (A,\pa) \to (B,\pa)$
and an isotopy class $\iB \in \Iso$ such that
$\Box( \iB) = (B,\pa)$, there exists a unique arrow
$\phi : \iA \to \iB$ in $\Iso$ such that $\Box(\phi) = [f]$. Let us
prove the uniqueness.  Assume that $\phi',\phi''$ are two such
arrows, represented by \Ainfty-morphisms
\[
F' : (A,\pa,\bfmu') \to (B,\pa,\bfnu')\ \hbox{ resp.} \ \ F'' : (A,\pa,\bfmu'') \to
(B,\pa,\bfnu'').
\]
The linear  parts $f'$ resp.~$f''$ of $F'$ resp.~$F''$ are chain homotopic 
to $f : (A,\pa) \to (B,\pa)$,  therefore $f' \sim f''$. Moreover,
$(B,\pa,\bfnu')$ is isotopic to $(B,\pa,\bfnu'')$, via an isotopy $T$
in the diagram
\begin{equation}
\label{Za chvili pojedu za Jarkou.}
\xymatrix{(A,\pa,\bfmu') \ar[r]^{F'}    &  (B,\pa,\bfnu')  \ar[d]^T 
\\
(A,\pa,\bfmu'')  \ar[r]^{F''}   & (B,\pa,\bfnu'').
}
\end{equation}
We are thus in the situation of item (ii) of Proposition~\ref{Pojedu tento
  tyden do Mercina?}, so $F'$ and $F''$ belong to the same equivalence
class. 

Let us prove the existence of a lift. Assume that $\iB$ is the isotopy
class of some $(B,\pa,\bfnu)$. Choose a homotopy inverse $g:(B,\pa)
\to (A,\pa)$ of $f$ and invoke Theorem~\ref{Stal se ten zazrak?} with the
roles of $f$ and $g$ reversed. We obtain \Ainfty-morphisms 
$G$ and $F$ in the diagram
\[
\xymatrix@C=3em@R=1.5em{ 
(A,\pa,\bfmu) \ar@{|->}[dd]_\Box
\ar@/^1em/@{-->}[rr]^{\redukce{\scriptsize $F$}}  
\ar@{<--}@/_1em/[rr]^{{G}} &&  
(B,\pa,\bfnu) \ar@{|->}[dd]^\Box
\\
&&
\\
(A,\pa)    \ar@/^1em/[rr]^f   \ar@{<--}@/_1em/[rr]^g    & & (B,\pa).
}
\]
The class of the \Ainfty-morphism $F$ is clearly  the required lift.

The fibration property can be established similarly. The
uniqueness leads, instead of~\eqref{Za chvili pojedu za Jarkou.}, to the diagram
\[
\xymatrix{(A,\pa,\bfmu') \ar[r]^{F'} \ar[d]_S    &  (B,\pa,\bfnu') 
\\
(A,\pa,\bfmu'')  \ar[r]^{F''}   & (B,\pa,\bfnu''),
}
\]
which is the situation of item (i) of Proposition~\ref{Pojedu tento
  tyden do Mercina?}. In the fibration case,
Corollary~\ref{hruba verze} suffices to construct, given
$(A,\pa,\bfmu)$, an \Ainfty-morphism $F$ extending  $f :(A,\pa) \to (B,\pa)$ 
as in 
\[
\xymatrix@C=3em@R=1.5em{ 
(A,\pa,\bfmu) \ar@{|->}[dd]_\Box
\ar@{-->}[rr]^{\redukce{\scriptsize $F$}}  
&&  
(B,\pa,\bfnu) \ar@{|->}[dd]^\Box
\\
&&
\\
(A,\pa)    \ar[rr]^f     && (B,\pa).
}
\]
The class of $F$ is then a lift of $[f]$. 
\end{proof}

Classically, there is a one-to-one correspondence
between presheaves  $R$  (i.e.\ contravariant functors to the category
$\Set$ of sets) on a category $\B$ and discrete fibrations $\ttE \to
\B$ with small fibers: 
\[
\big\{ \hbox { discrete fibrations over $\B$ with small fibers } \big\} 
\cong
\{\ \hbox {presheaves over $\B$ }\}.
\]
The category $\ttE$ is the `category of elements' of the functor
$R$, usually denoted by $\int_\B R$. 
In the situation of Theorem~\ref{Uz druhy vikend prsi.}, the
contravariant functor
$R: \CHE \to \Set$ assigns to 
a chain complex $(A,\pa)$ the set $R(A,\pa)$ of isotopy
classes of \Ainfty-algebras on~$(A,\pa)$. For a chain homotopy equivalence 
$f : (A,\pa) \to (B,\pa)$, the set map 
$R([f]) : R(B,\pa) \to  R(A,\pa)$ sends the isotopy class of an
\Ainfty-algebra $(B,\pa,\bfnu)$ to the unique isotopy class  of an
\Ainfty-algebra $(A,\pa,\bfmu)$ for which there exists an \Ainfty-morphism
$G : (B,\pa,\bfnu) \to (A,\pa,\bfmu)$ extending a chain
homotopy inverse of $f$. 

Similarly, there is  a
one-to-one correspondence between {\em covariant\/} functors $L:\B \to \Set$
and discrete opfibrations  $\E \to
\B$ with small fibers: 
\[
\big\{ \hbox { discrete opfibrations over $\B$ with small fibers } \big\} 
\cong
\{\ \hbox {covariant functors $\B \to \Set$ }\}.
\]
The category $\E$ corresponding to $L$ is usually denoted
by $\int^\B L$.
We leave the explicit description of the covariant functor  $L : \CHE \to
\Set$ corresponding to $\Box : \Iso \to \CHE$ as a simple exercise
for the reader. The following proposition is just a fancy
packing of already established results.

\begin{proposition}
There are isomorphisms of categories \
$
\Iso \cong \int_\CHE R \cong \int^\CHE L\ .
$    
\end{proposition}

\section{Generalization to strongly homotopy $\oP$-algebras}
\label{Medvidek s Micinkou}

Operads (resp.~cooperads) in this section will live in the monoidal
category of graded modules over a commutative unital ring $R$ which
will sometimes  be assumed to be a field of characteristic $0$. 
Operads (resp.~cooperads)
will be unital (resp.~counital) such that
$\oP(0)=0$, $\oP(1) \cong R$ (resp.~$\cooP(0)=0$, $\cooP(1) \cong R$).

\subsection*{Extending homotopies -- the non-$\Sigma$ case}

Our wonderful proof of Lemma~\ref{Vcera jsem ten kopec na Tocnou
  vyjel.} made extensive use of the fact that \Ainfty-homotopies can be
described  as (co)derivation homotopies between dg coalgebra
morphisms. In this subsection we show that the same approach
applies also to $\oP_\infty$-algebras, with  $\oP$  a
{\/\em nonsymmetric\/} quadratic Koszul operad $\oP$.

For a non-symmetric (non-$\Sigma$ for short)  cooperad $\cooP$ and a~(graded)
$R$-module $V$ denote by $\CC V$ the 
cofree counital conilpotent $\cooP$-coalgebra cogenerated by $V$.
As a (graded) $R$-module, 
\begin{equation}
\label{Jituska s Jitulinkou}
\CC V =  \bigoplus_{n \geq 1} \CCn nV,\ \hbox { where } 
\CCn nV :=  \cooP(n) \ot
\otexp Vn.
\end{equation} 
We will be interested in objects of the form $(\CC V,\delta)$, with 
$\delta$ a degree $-1$ coderivation that squares to zero. 
Such a $\delta$ is determined by degree $-1$ linear maps $\delta_n :
\CCn nV \to V$, $n \geq 1$. The condition $\delta^2 = 0$ implies that $\delta_1
: \cooP(1) \ot V \cong V \to V$ is a differential, so $(V,\delta_1)$
is a chain complex. 

Similarly, a dg coalgebra morphism $\phi : (\CC {V'},\delta') \to
(\CC{V''},\delta'')$ is determined by a sequence of degree $0$ linear
maps $\phi_n : \CCn nV \to V$, $n \geq 1$. The condition $\delta''\phi = \phi\delta'$ implies
that $\phi_1 : (V',\delta'_1) \to  (V'',\delta''_1)$ is a chain map.
A coderivation homotopy between $\phi$  and
another dg coalgebra map  $\psi : (\CC {V'},\delta') \to
(\CC{V''},\delta'')$ is a degree $+1$ linear map $\eta : \CC {V'} \to
\CC {V''}$ such that
\begin{subequations}
\begin{equation}
\label{Vecer jedeme}
\psi - \phi=  \delta'' \eta + \eta  \delta'
\end{equation}
and, for $n \geq 1$ and an arbitrary structure operation  $\Delta \in \cooP(n)$,
 \begin{equation}
\label{s Jarkou na chalupu.}
\Delta \eta = \sum_{i=1}^n   (\otexp \phi{i-1}\ot \eta \ot 
\otexp \psi {n-i})\Delta \ .
\end{equation}
\end{subequations}
Such an $\eta$ is
determined by a family $\eta_n : \CCn nV \to V$, $n \geq 1$ 
of degree $+1$ linear maps. The following proposition generalizes the
trick used in the proof of Lemma~\ref{Vcera jsem ten kopec na Tocnou vyjel.}.

\begin{proposition}
\label{Cas se krati.}
Let $\phi : (\CC {V'},\delta') \to  (\CC {V''},\delta'')$ be a morphism
of dg $\cooP$-coalgebras. An arbitrary  chain map
$\psi_1 : (V',\pa') \to (V'',\pa'')$, chain
homotopic to the linear part $\phi_1$ of $\phi$ via a chain homotopy
$\eta_1$, can be extended to a morphism $\psi : (\CC
{V'},\delta') \to  (\CC {V''},\delta'')$ of dg
$\cooP$-coalgebras,
homotopic to $\phi$ via a coderivation homotopy $\eta$
extending $\eta_1$. 
\end{proposition}

\begin{proof}
Choose degree $+1$ linear maps $\eta_n : \CCn n{V'} \to
V''$, $n \geq 2$, and extend the family $\{\eta_n\}_{n \geq 1}$ to a
  linear map $\eta : \CC{V'} \to \CC{V''}$ satisfying~\eqref{s Jarkou na
    chalupu.}. In the next step we will extend $\psi_1$ to a morphism
  $\psi :\CC
  {V'} \to \CC{V''}$ such that~\eqref{Vecer jedeme} holds. Let
\[
J(n) := \{k,\Rada r1k\ |\ k \geq 2,\   
      r_1,\ldots, r_k \geq 1,\
r_1+ \cdots + r_k = n\}
\] 
and, for $\kappa \in J(n)$, denote by $\cogamma_\kappa : \cooP(n) \to \cooP(k) \ot
\cooP(r_1) \ot \cdots \ot \cooP(r_k)$ the corresponding structure
operation of the cooperad $\cooP$.
Given linear maps $u_i : \CCn {r_i}{V'} \to V''$, $1 \leq i \leq k$, and \hbox{$\omega :
\CCn k{V'} \to V''$}, define
$\Upsilon_\kappa(\omega;\Rada u1n) :\CCn n{V'}
\to V''$ 
to be the composite 
\begin{align*}
&\xymatrix@1@C=1.3em{\CCn n{V'} \
\ar[rr]^(.28){\cogamma_\kappa \ot \otexp{\id_{V'}} n}\ && \
\ \cooP(k) \ot \cooP(r_1)\ot \cdots \ot \cooP(r_k)  \ot \otexp {V'}n \ar[rr]^(.85)\cong
&&
}
\\
& \hskip 4em   \xymatrix{\ar[r]^(.15)\cong
& \cooP(k) \ot \CCn {r_1}{V'} \ot   \cdots \ot \CCn {r_k}{V'} \ar[rr]^(.68){\id_{\ssscooP(k)} \ot
u_1\ot \cdots \ot u_r}
&& \CCn k{V'}   \ar[r]^(.6)\omega  &    V''.
}
\end{align*}
Likewise, for $2 \leq k \leq n$ and $1 \leq i \leq k$, denote by
\[
\cogamma^k_{n,i} : \cooP(n) \to \cooP(k) \ot \otexp{\cooP(1)}{i-1} \ot
\cooP(l) \ot \otexp{\cooP(1)}{k-i} \cong \cooP(k) \ot \cooP(l),\  l=
k+n-1,
\]
the corresponding structure operation.
For linear maps $\sigma : \CCn k{V'} \to V''$ and $v : \CCn l{V'} \to V''$,
let  $\Gamma^k_{n,i} (\sigma; v) : \CCn  n{V'} \to V''$ be the composite
\[
\begin{aligned}
&\xymatrix@1{
\CCn n{V'} \ \ar[rr]^(.38){\cogamma^k_{n,i} \ot \id_{V'}^{\ot n}} 
&& \ \cooP(k) \ot \cooP(l) \ot \otexp {V'}n \ \ar[r]^(.8)\cong&
}
\\
& \hskip 4em   
\xymatrix@1{\ar[r]^(.15)\cong
 & \
\cooP(k) \ot \otexp {V'}{i-1} \ot \CCn l{V'}  \ot 
\otexp {V'}{k-i} \ 
\ar[rrr]^(.68){\id_{\ssscooP(k)}  \ot \id_{V'}^{\ot i-1} \ot v \ot \id_{V'}^{\ot k-i} }
&&& \ \CCn k{V'} \ar[r]^(.65)\sigma & V''.
}
\end{aligned}
\]
Equation~\eqref{Vecer jedeme} is equivalent to the following system of equations
\begin{gather*}
\tag{$\eta_n$}
\psi_n  =  \ d (\eta_n)  +\phi_n
+  \sum_{\kappa \in J(n)}\ \sum_{i= 1}^k \pm\Upsilon_\kappa
(\delta''_k; \psi_{r_1},\ldots,
\psi_{r_{i-1}}  ,  \eta_{r_i}, \phi_{r_{i+1}}, \ldots,\phi_{r_k})
\\
\nonumber 
+   \sum_{k=2}^n \  \sum_{1= i}^k \pm
\Lambda^k_{n,i}(\eta_k; \delta'_l),
\ n \geq 1,
\end{gather*}
in which $d(-)$ denotes the induced differential on the space of linear
maps $\CCn n{V'} \to V''$. As in the proof of Lemma~\ref{Vcera jsem ten
  kopec na Tocnou vyjel.}, the system~$(\eta_n)$ determines
inductively  the higher
components of $\psi$. It remains to show that $\psi$
commutes with the differentials. This is done in exactly the same way
as in the last paragraph of the proof of  Lemma~\ref{Vcera jsem ten
  kopec na Tocnou vyjel.}.
\end{proof}

Let $\oP$ be an operad over a field $\bbk$ of characteristic $0$. 
In~\cite{zebrulka},
$\oP$-infinity, or $\oP_\infty$-algebras for short, were defined as
algebras for the minimal model of $\oP$. If $\oP$ is quadratic Koszul,
they are the same as dg coalgebras of the form $(\CC{\susp A},\delta)$,
with $\delta$ a degree $-1$ coderivation that squares to zero, where
$\cooP := \oP^\antishriek$, the Koszul dual cooperad of
$\oP$~\cite[\S 10.1.1]{loday-vallette}. 
$\oP_\infty$-algebras, their
morphism and homotopies look similarly as the
corresponding $A_\infty$-notions in the
introduction. Namely, $\oP_\infty$-algebras are
objects $(A,\pa,\bfmu) =(A,\partial,\mu_2',\mu_3',\ldots)$, where
$(A,\pa)$ is a chain complex and \hbox{$\mu_n: \cooP(n)\! \ot\!
  \!\otexp An \!\to\! A$}, $n \geq 2$, are
$R$-linear maps of degree \hbox{$n\!-\!2$}. 
Their morphism are given by sequences
$F = (f,f_2,f_3,\ldots)$, with  $f$ a chain map and \hbox{$f_n :
\cooP(n) 
\ot \otexp {A'}n \to A''$} $R$-linear degree $n\!-\!1$ maps. Finally,
homotopies are given by sequences $H = (h,h_2,h_3,\ldots)$, where $h$
is a chain homotopy and  \hbox{$h_n :
\cooP(n) \! \ot \!\! \otexp {A'}n \to A''$} are degree~$n$ $R$-linear
maps.

In situations specified in the following
definition, the notion of
$\oP_\infty$-algebras makes sense over the ring of integers $\bbZ$,
and thus for an arbitrary commutative unital ring~$R$.

\begin{definition}
\label{Uz to neni ono.}
Suppose that there exists a non-$\Sigma$ operad $\oP_{\Se}
^!$ in the category of sets such that $\oP^\antishriek =
{\rm Map}_{\ \Se}(\oP_\Se^!,\bbk)$, with the obvious cooperad
structure. Taking $\cooP : =
{\rm Map}_{\ \Se}(\oP_\Se^!,\bbZ)$, we define {\em homotopy $\oP$-algebras
over $\bbZ$\/}, or simply {\/\em $\oP^\bbZ_\infty$-algebras\/}, 
as dg $\cooP$-coalgebras in the category of Abelian groups
of the form $(\CC{\susp A},\delta)$.  By overall tensoring with $R$ we
obtain the category of {\/\em $\oP^R_\infty$-algebras\/}  for an
arbitrary unital commutative ring $R$.
\end{definition}

The assumption of Definition~\ref{Uz to neni ono.} is satisfied
e.g.~for the operad $\Ass$ governing associative algebras when 
we are in the situation of
the previous sections which addressed $\Ass_\infty^R$-algebras. The interested reader can find other
cases where  Definition~\ref{Uz to neni ono.} applies in~\cite{encyclopedia}.

\noindent {\bf Conclusion.}
Proposition~\ref{Cas se krati.} with $V' := \susp A'$, $V'' := \susp A''$, gives a
{verbatim} g{eneralization}
of Lemma~\ref{Vcera jsem ten kopec na Tocnou vyjel.} and
Corollary~\ref{Bude dnes opravene kolo?} to $\oP_\infty$-algebras if
$\oP$ is a non-$\Sigma$ quadratic Koszul operad
over a field of characteristic~$0$. If the assumptions of
Definition~\ref{Uz to neni ono.} are fulfilled,
Lemma~\ref{Vcera jsem ten kopec na Tocnou vyjel.} and
Corollary~\ref{Bude dnes opravene kolo?}
hold also for $\oP_\infty^R$-algebras.

\subsection*{Extending homotopies -- the $\Sigma$-case} 
Let $\oP$ be a quadratic Koszul operad over a field $\bbk$ of characteristic $0$, and
$\cooP := \oP^\antishriek$. As before, $\oP_\infty$-algebras will be
dg coalgebras of the form \hbox{$(\CC{\susp A},\delta)$}. In terms of
multilinear operations, $\oP_\infty$-algebras and their
homomorphisms look similar as in the non-$\Sigma$ case, except
that $\CCn nV$ in~\eqref{Jituska s Jitulinkou} must be replaced by the subspace
\hbox{$(\cooP(n)\! \ot \!\!\otexp Vn)^{\Sigma_n}$} of $\Sigma_n$-invariants.
However, homotopies between morphisms are not given by
coderivation homotopies, but by a suitable path or
cylinder objects in the related model structure. It was proven
in~\cite{tale} that, for Koszul
quadratic $\oP$, the correct homotopy can
be represented as an algebra for the cofibrant two-colored
operad $\minim_{\malesipky}$ described in~\cite[Theorem~18]{ws01}.

Let us recall the necessary definitions, referring
to~\cite[Section~4]{ws01} for details. If  $\minim   = (\calF(X),d)$
is the minimal model of
$\oP$, where $\calF(-)$ denotes the free operad functor, 
a $\oP_\infty$-algebra is the same as a morphism $\minim \to
\End_A$ to the endomorphism operad of
$A$. To describe $\oP_\infty$-morphisms, we need the two-colored operad
\[
\minim_\sipka = (\calF(X_\B; p; X^p; X_\W),D)
\]
of~\cite[Theorem~7]{ws01}, generated by two copies $X_\B$ and $X_\W$
of $X$ for $\oP_\infty$-algebras sitting on $A'$ resp.~$A''$, the
generator $p$ for the linear part $f$ of $F$, and the shifted $X$ for
its higher parts. A $\oP_\infty$-morphism is then a dg operad morphism 
$\minim^p_\sipka \to \End_{A',A''}$ to the endomorphism operad of the
two-colored chain complex $\{A',A''\}$.
A $\oP_\infty$-homotopy between  $F$ and $G$ is a
morphism $\rho : \minim_{\malesipky} \to  \End_{A',A''}$, 
where
\[
\minim_{\malesipky} = (\calF(X_\B; p,q,h; X^p,X^q,X^h; X_\W),D)
\]
is the two colored operad introduced in~\cite[Theorem~18]{ws01}. It is
a cellular extension of $\minim_\sipka$, with $\{q,X^q\}$ the generators
for the components of $G$ and $\{h,X^h\}$ the generators for the
components of the $\oP_\infty$-homotopy added. Very crucially,  
$D(h) = q-p$. Consider finally the dg suboperad
\[
\minim^{p,h}_\sipka = (\calF(X_\B; p,q,h; X^p; X_\W),D) 
\]
of $\minim_{\malesipky}$. Algebras for $\minim^{p,h}_\sipka$ clearly consist
of a $\oP_\infty$-morphism $F$, a chain map $g$ and a chain 
homotopy between the linear part $f$ of $F$ and $g$. 

\begin{proposition}
\label{Voskova panenka a certik Pik} 
Each morphism $\alpha : \minim^{p,h}_\sipka \to \oA$ of two-colored dg
operads in the diagram
\[
\xymatrix@C=3em{\minim_{\malesipky} \ar@{-->}[r]^\rho & \oA \ar@{=}[d]
\\
\minim^{p,h}_\sipka \ar[ur]^\alpha\ar[r]  \ar@{^{(}->}[u]^\iota   &\oA
}
\]
in which $\iota$ is an inclusion of dg colored operads, 
lifts to a morphism $\rho  : \minim_{\malesipky} \to \oA$.
\end{proposition}

\begin{proof}
Notice that $\minim^{p,h}_\sipka$ is a
coproduct of $\minim_\sipka$
with the acyclic dg operad $(\calF(p,q,h),D)$, therefore 
the natural map $\minim^{p,h} \to
\minim_\sipka$ is a quism, and so is the composite
\[
\minim^{p,h} \stackrel\iota\hookrightarrow \minim_{\malesipky} 
\longrightarrow \minim_\sipka.
\]
Since the second map is
a quism by~\cite[Theorem~18]{ws01}, the inclusion $\iota$ is a quism
too.
The proof is finished by the standard obstruction theory using the fact that
$\minim_{\malesipky}$ is an acyclic cellular extension of $\minim^{p,h}_\sipka$.
\end{proof}

\noindent 
{\bf Conclusion.}
Proposition~\ref{Voskova panenka a certik Pik} with $\oA := \End_{A',A''}$ 
gives a {verbatim} generalization
of Lemma~\ref{Vcera jsem ten kopec na Tocnou vyjel.} and
Corollary~\ref{Bude dnes opravene kolo?} to $\oP_\infty$-algebras for
$\oP$ a quadratic Koszul operad 
over a field of characteristic $0$.

\subsection*{Transfers of \/$\oP_\infty$-structures.}
The ground ring  in this subsection will be a
field of characteristic~$0$, though in some situations  the main result
may still hold over an arbitrary commutative unital ring~$R$,
cf.~Remark~\ref{Kasparek s Myskou}. Let
$\oP$  be a quadratic Koszul operad and $\minim   = (\calF(X),d)$
its minimal model. 
The central role will be played by the two-colored
dg-suboperad $\minimS$ of the operad $\minimiso$
in Theorem~24 of~\cite{ws01} generated by 
$X_\B,X_W,f_0,f_1,g_0,X^{f_0},X^{f_1}$ and $X^{g_0}$.
Modifying the arguments in the proof of that theorem we can show that
$\minimS$ is a cellular resolution of the two-colored operad~$\oPS$
whose algebras consist of $\oP$-algebra structures on chain complexes
$A'$ and $A''$, a $\oP$-algebra morphism $f:A' \to A''$ and its
inverse $g : A'' \to A'$. 
Algebras for $\minimS$ therefore consist of $\oP_\infty$-structures on
$A'$ and $A''$, $\oP_\infty$-morphisms $F :A'
\to A''$ and $G: A'' \to A'$, and a $\oP_\infty$-homotopy $H$ between
$GF$ and $\id_{A'}$.

Let $\minimuS$  be the coproduct of the operad $\minim$
with the two-colored dg operad $(\calF(f,g,h),D)$, where $f: \B \to \W$, $g:
\W \to \B$ are degree $0$ generators and $h : \B \to \B$ a degree $1$
generator such that $gf=D(h)$. Algebras for
$\minimuS$ consist of a $\oP_\infty$-structure on 
$A'$, and chain maps $f:A' \to A''$, $g:A''
\to A'$ such that $h$ is a chain homotopy between $gf$ and
$\id_{A'}$. The operad $\minimuS$
resolves the operad
$\oPuS$ that describes structures consisting of a $\oP$-algebra on
$A'$ a chain map  $f:A' \to A''$ and its left inverse $g:A''
\to A'$.

As noticed in~\cite[Example~12]{haha}, 
there exists a morphism $\psi : \oPS \to
\oPuS$ such that $\psi(p) := p$ for $p \in \oP$ siting in the
\B-color, $\psi(f) := f$, $\psi(g) := g$ and
$\psi (p) := fp \otexp gn$  for $p \in \oP(n)$, $n \geq 2$,
  siting in the \W-color.  
By the standard properties of cellular
resolutions, the map $\psi$ lifts into a morphism $\widetilde \psi :
\minimS \to \minimuS$. Each  $\minimuS$-algebra is thus also an
$\minimS$-algebra. 

The {\bf conclusion} of this subsection is so important that we
formulate it as the following theorem where $\oP$ is a quadratic
Koszul operad over a field of characteristic $0$.

\begin{theorem}
\label{Alicek, Mikinka a Flicek}
Let $f : (A,\pa) \to (B,\pa)$ and  $g : (B,\pa) \to
(A,\pa)$ be chain maps and $h$ a chain homotopy
between $gf$ and $\id_A$.
Then any $\oP_\infty$-algebra structure $(A,\pa,\bfnu)$ on $(A,\pa)$ induces
a $\oP_\infty$-algebra structure $(B,\pa,\bfnu)$ on
$(B,\pa)$ such that the chain maps $f$ and $g$ extend to
$\oP_\infty$-morphisms $F$ and~$G$, and the chain homotopy $h$ extends to
a $\oP_\infty$-homotopy  between $GF$ and $\id_A$.
\end{theorem}

\begin{remark}
\label{Kasparek s Myskou}
We believe that, for non-$\Sigma$ operads
satisfying the conditions of Definition~\ref{Uz to neni ono.},  
there exist explicit formulas for transfers of $\oP_\infty^\bbZ$-structures 
involving $\bbZ$-linear combinations of decorated trees, parallel to
the formulas in~\cite{tr}. 
This would imply that for such operads the analogs of  Theorem~\ref{Stal se ten
  zazrak?} and Corollary~\ref{hruba verze} hold
over an arbitrary unital commutative ring.
\end{remark}

\section*{Concluding remarks}

Section~\ref{Kralicek a Oslicek} hinges of Corollary~\ref{Bude dnes
  opravene kolo?} and Section~\ref{dva kasparci a certik Pik} moreover
uses  Theorem~\ref{Stal se ten zazrak?} and Lemma~\ref{Vcera jsem ten
  kopec na Tocnou vyjel.}. As we proved in Section~\ref{Medvidek s
  Micinkou}, the obvious analogs of these results hold for quadratic
Koszul operads over a field of characteristic $0$, and thus all the
other results of the paper hold too.
We strongly believe
that, for non-$\Sigma$-operads satisfying the condition of
Definition~\ref{Uz to neni ono.}, the results of this
article hold over an arbitrary commutative unital ring $R$.

\end{document}